\documentclass[11pt]{amsart}

\usepackage{amsmath}
\usepackage{amssymb}
\usepackage{amsthm}

\usepackage{times}
\usepackage[utf8]{inputenc}
\usepackage[T1]{fontenc}
\usepackage{tikz-cd}

\newtheorem{thm}{Theorem}[section]
\newtheorem{cor}[thm]{Corollary}
\newtheorem{lem}[thm]{Lemma}
\newtheorem*{thm*}{Theorem}

\newtheorem{prob}[thm]{Question}
\newtheorem{prop}[thm]{Proposition}

\theoremstyle{definition}
\newtheorem{defin}[thm]{Definition}

\DeclareMathOperator{\acc}{acc}
\DeclareMathOperator{\dens}{dens}

\DeclareMathOperator{\kappatree}{^{<\omega}\kappa}
\DeclareMathOperator{\Bool}{Bool}
\DeclareMathOperator{\BoolSeq}{BoolSeq}
\DeclareMathOperator{\B}{\mathcal B}
\DeclareMathOperator{\K}{\mathcal K}
\DeclareMathOperator{\E}{\mathcal E}

\newcommand{\aut}{\operatorname{Aut}}
\newcommand{\blim}{\mathbb{B}_\lambda}
\newcommand{\rg}{\operatorname{rg}}
\newcommand{\dom}{\operatorname{dom}}
\newcommand{\s}{\subseteq}
\newcommand{\col}{\operatorname{Coll}}
\newcommand{\coll}{\operatorname{Coll}(\omega,\kappa)}

\newcommand{\collkappa}{\operatorname{Coll}(\omega,\kappa)}
\newcommand{\colllambda}{\operatorname{Coll}(\omega,<\lambda)}

\title{Homogeneity of the L\'evy collapse from the perspective of Fra\"iss\'e theory}
\author{Ziemowit Kostana}
\address{Institute of Mathematics of the Czech Academy of Sciences, \v{Z}itn\'{a} 25, 115~67 Prague 1, Czech Republic}
\address{Wroc\l{}aw University of Science and Technology, Wybrze\c{z}e Wyspia\'nskiego 27, 50~370 Wroc\l{}aw, Poland}
\email{ziemowit.kostana@pwr.edu.pl}

\begin{document}

\maketitle

\begin{abstract}
Given a strongly inaccessible cardinal $\lambda$, we study the Fra\"iss\'e class of all Boolean algebras of size $<\lambda$, together with regular embeddings. We prove that this is indeed a Fra\"iss\'e class, and its limit has the same completion as the L\'evy collapse. We also give a direct proof that the collapsing algebra of density $\kappa$ is not the union of a $\kappa$-chain of regular sub-algebras of density $<\kappa$.
\end{abstract}



\section{Introduction}
\subsection{Fra\"iss\'e-J\'onsson theory}

The classical Fra\"iss\'e theory studies relations between certain classes of finite structures, called \emph{Fra\"iss\'e classes}, and countable structures that are homogeneous with respect to finite substructures. However, since the original paper of Fra\"iss\'e \cite{fraisse}, the theory received a large number of far-reaching generalizations. The category-theoretic framework for Fra\"iss\'e theory was introduced by Droste and G\"obel \cite{dg1} \cite{dg2}, and was further generalized by Kubi\'s \cite{kub}. The theory originally developed in the language of models has became an influential technology in general topology \cite{solecki}, topological dynamics \cite{kpt}, theory of Polish groups \cite{truss} \cite{kpt}, or functional analysis \cite{garbulinska} \cite{espana}. The article \cite{macpherson} is an excellent survey of what is known about homogeneous structures and their automorphism groups, and \cite{kechris-rosendal} contains an extensive introduction to the study of topological properties of these groups.

The theory of uncountable Fra\"iss\'e limits is often called \emph{Fra\"iss\'e-J\'onsson theory}, as J\'onsson was the first to generalize the original work of Fra\"iss\'e to the uncountable case. Perhaps the most remarkable instance of Fra\"iss\'e-J\'onsson theory is Parovi\v{c}enko Theorem, stating that under Continuum Hypothesis $\mathcal{P}(\omega)/\operatorname{Fin}$ is the unique Boolean algebra of size continuum that has Strong Countable Separation Property \cite{parovicenko}.

A streamlined, model-theoretic exposition of the classical Fra\"iss\'e theory can be found in \cite{hodges}.

\subsection{Collapsing algebras}

Given an uncountable cardinal $\kappa$, the \emph{collapsing algebra} $\coll$ is the unique complete Boolean algebra of density $\kappa$ that collapses $\kappa$ to a countable ordinal. It has many interesting properties. For example, by a theorem of Solovay, it is countably generated as a complete Boolean algebra (\cite{bool}, p.191), showing that countably generated complete Boolean algebras can be arbitrarily large. Also, every Boolean algebra embeds into some collapsing algebra as a regular sub-algebra. Any collapsing algebra can be characterized up to isomorphism by certain non-distributivity condition. This characterization, and some further properties can be found in \cite{bool}.

\subsection{L\'evy collapse}

The \emph{L\'evy collapse} is a classical forcing notion, used for collapsing strongly inaccessible cardinals to "small" cardinals, like $\omega_1$ or $\omega_2$. One of the most spectacular effects of L\'evy-collapsing a strongly inaccessible cardinal to $\omega_1$ is making all projective sets Lebesgue-measurable \cite{b-j}. Collapsing a strongly inaccessible cardinal to $\omega_2$ kills all Kurepa trees (\cite{silver},\cite{jech} Thm. 27.9). Somewhat less known application of the L\'evy collapse is forcing that every compact Hausdorff space of weight $\omega_1$ has size $\omega_1$ or $2^{\omega_1}$, although in this case the forcing is the product of the L\'evy collapse with the generalized Cohen forcing \cite{juhas}.

Given a strongly inaccessible cardinal $\lambda$, we study the class of all Boolean algebras of size less than $\lambda$ -- equipped with regular embeddings -- from the perspective of Fra\"iss\'e-J\'onsson theory. We denote this class by $\Bool_\lambda$. We also consider the class $\BoolSeq_\lambda$, consisting of all Boolean algebras that are unions of increasing $\lambda$-chains of regular sub-algebras. For precise definitions, see Definition \ref{classes}. The main results of the paper are:
\begin{enumerate}
    \item[A.] $\Bool_\lambda$ is a Fra\"iss\'e class, and its limit has the same completion as the L\'evy collapse of $\lambda$ to $\omega_1$. 
    \item[B.] The automorphism group of the L\'evy collapse is a universal group in the class of groups $\aut(B)$, for $B \in \BoolSeq_\lambda$. 
\end{enumerate}
We also give a new proof of the fact that $\coll$ is not in the class $\BoolSeq_\kappa$. This fact itself is rather standard, but the argument relies on the preservation of $\kappa$-chain condition by finite support iterations. We give a direct proof that still might be of some interest.

\section{Main definitions and notation}

The paper is self-contained with respect to Fra\"iss\'e theory. Formulations of most of the theorems base only on basic model-theoretic notions, however proofs often rely on forcing arguments. All relevant notions from the theory of Boolean algebras can be found in the following short introduction, or in \cite{bool}.

\subsection{Partial orders}

All partial orders under consideration are separative. For a partial order $P$, we denote by $\overline P$ the \emph{completion} of $P$ -- the unique complete Boolean algebra containing $P$ as a dense subset. The completion of $P$ can be realized as the Boolean algebra of regular open subsets of $P$, in some specific topology. The former universal property is however more relevant, than what the elements of $\overline P$ are specifically.

\subsection{Regular Boolean sub-algebras}

A subset $R \s A$ of a Boolean algebra $A$ is \emph{dense} if for every $a \in A\setminus \{0\}$ there exists $r \in R\setminus \{0\}$ such that $r \leq a$. The least size of a dense subset of $A$ is called the \emph{density} of $A$, and denoted $\dens(A)$.

A sub-algebra $A\subseteq B$ is \emph{regular} if every maximal antichain in $A$ is also a maximal antichain in $B$. We write $A \sqsubseteq B$. An embedding of Boolean algebras is called a \emph{regular embedding} if its image is a regular sub-algebra. Regular embeddings are similar to \emph{complete embeddings}. An embedding $e:A\longrightarrow B$ is \emph{complete} if $A$ is a complete algebra, and for each subset $S \subseteq A$
$$\bigvee e[S]= e[\bigvee S], \; \bigwedge e[S]=e[\bigwedge S].$$

The difference is that the expressions $\bigvee S$ and $\bigwedge S$ might be undefined in non-complete algebras. We define a regular suborder of a partial order the same way, and also write $P \sqsubseteq R$. We will be frequently applying the following simple and well-known consequence of the Sikorski's Extension Criterion.

\begin{prop}\label{extension}
	Any regular embedding $e:A\longrightarrow B$ extends uniquely to a complete embedding
	$\bar{e}:\overline{A}\longrightarrow \overline{B}$.
\end{prop}

\subsection{Forcing arguments}

In the proof of Theorem \ref{mainthm} we are using methods from the forcing theory. They are similar to those used in the proofs of Lemma 26.7 and Lemma 26.9 in \cite{jech}. 

\par For a Boolean algebra $B$, we set $B^+=B \setminus \{0\}$. In the theory of forcing, we commonly consider embeddings $$B^+ \longrightarrow B^+\times C^+,$$
where $B$ and $C$ are Boolean algebras, given by
$$b \mapsto (b,1_C).$$

This is a regular embedding of partial orders, however it does \emph{not} induce a regular embedding of Boolean algebras: even if $b$ and $b'$ are incompatible in $B$, we have 
$$(0_B,1_C)\leq (b,1_C),(b',1_C).$$

\begin{prop}
For Boolean algebras $B$ and $C$ there is a natural dense embedding $i:B^+\times C^+ \rightarrow B\oplus C$. In particular
$$\overline{B^+\times C^+}\simeq \overline{B \oplus C},$$
where $\oplus$ is the free product. 
\end{prop}
\begin{proof}
By the Stone duality we can find compact zero-dimensional spaces $X_B$ and $X_C$ such that $B$ and $C$ are algebras of clopen sets on $X_B$ and $X_C$ respectively. Evidently $$B^+=\operatorname{clop}(X_B)\setminus \{\emptyset \}$$
and
$$C^+=\operatorname{clop}(X_C)\setminus \{\emptyset \}.$$
We define the embedding 
$$i:\operatorname{clop}(X_B)^+\times \operatorname{clop}(X_C)^+ \rightarrow \operatorname{clop}(X_B\times X_C)$$
by $i(p,q)=x\times y$. 
    
\end{proof}

\begin{defin}
The \emph{canonical embedding of $B^+$  into $\overline{B \oplus C}$} is the composition
$$B^+\longrightarrow B^+\times C^+ \sqsubseteq \overline{B^+\times C^+}=\overline{B \oplus C},$$
where the first arrow maps $b$ to $(b,1_C)$.
\end{defin}
By Proposition \ref{extension}, the canonical embedding of $B^+$ into $\overline{B}$ extends uniquely to the \emph{canonical embedding of $\overline{B}$ into $\overline{B \oplus C}$}. We refer for the details to \cite{bool}.

\section{Fra\"iss\'e-J\'onsson Theory}

In this section we collect basic notions concerning Fra\"iss\'e limits, and state the main theorem connecting Fra\"iss\'e classes with Fra\"iss\'e limits. The main theorem of this section can be stated in many different levels of generality. Perhaps the most general formulation is Corollary 4.7 in \cite{kub}, while the most basic one in due to Fra\"iss\'e \cite{fraisse}.  The formulation we have chosen is rather intermediate, suitable for our applications.

\subsection{Fra\"iss\'e-J\'onsson classes}

We consider a class $\K$, consisting of structures in some first-order language, together with a distinguished class of embeddings, denoted by $\E$. We assume that $\E$ contains all isomorphisms and is closed under compositions. For $A,B \in \K$ we write $A\sqsubseteq_{\E}B$ if $A\subseteq B$ and the inclusion map belongs to $\E$.

\begin{defin}
    For a limit ordinal $\delta$, an $\sqsubseteq_{\E}$-increasing sequence $\{A_\alpha \mid \alpha<\delta\}$ of $\K$-structures is \emph{continuous} if for every limit $\beta<\delta$ we have
    $$A_\beta=\bigcup\limits_{\alpha<\beta}A_\alpha.$$
    
\end{defin}

\begin{defin}
    For an uncountable cardinal $\kappa$, the class $(\K,\E)$ is \emph{$<\kappa$-closed} if whenever $\{A_\alpha \mid \alpha<\delta\}$ is a continuous sequence of $\K$-structures, and $\delta<\kappa$, then

    $$A_\delta:=\bigcup_{\alpha<\delta}A_\alpha \in \K,$$
    and $A_\alpha \sqsubseteq_{\E} A_\delta$ for every $\alpha<\delta$.    
\end{defin}

\begin{defin} For a class $(\K,\E)$ we will say that:
	\begin{itemize}
		\item $(\K,\E)$ has the \emph{Joint Embedding Property} ($\operatorname{JEP}$), if for each $a,b \in \mathcal{K}$ there exists $\E$-embeddings of $a$ into $c$ and of $b$ into $c$, for some $c \in \K$.\\
	\begin{center}
	\begin{tikzcd}
		& a  \ar[dr]
		&
		& \\
		& 
		&c
		&
		&
		&
		\\
		& b  \ar[ur]
		&
		& 
	\end{tikzcd}
	\end{center}
		\item $(\K,\E)$ has the \emph{Amalgamation Property} ($\operatorname{AP}$), if for each pair of embeddings $f:a\longrightarrow b$, $g:a\longrightarrow c$ from $\E$ there exists $d\in \K$, together with a pair of embeddings $f':b\longrightarrow d$, $g':c\longrightarrow d$ in $\E$, such that $f'\circ f = g'\circ g$.\\
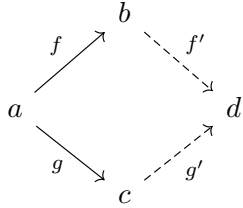
		\begin{center}
		\begin{tikzcd}
			& b  \ar[dr,dashed, "f'"]
			&
			& \\
			a \ar[ur, "f"] \ar[dr, swap, "g"] 
			& 
			&d 
			&
			&
			&
			\\
			& c  \ar[ur, dashed, swap, "g'"]
			&
			& 
		\end{tikzcd}
		\end{center}
	\end{itemize}
\end{defin}

\begin{defin}A structure $\mathbb A$ is:
	\begin{itemize}
		\item \emph{$(\K,\E)$-universal}, if for every $B \in \K$, there exists an $\E$-embedding $B \longrightarrow \mathbb A$,
		\item \emph{$(\K,\E)$-injective}, if for any
        $\E$-embedding $f:B\longrightarrow \mathbb A$, any $C \in \K$ such that $B\sqsubseteq_{\E} C$, there exists an extension of $f$ to an $\E$-embedding $\overline{f}:C\longrightarrow \mathbb A$.
        
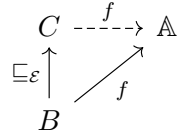
		\begin{center}
		\begin{tikzcd}
		C\arrow[r, dashed, "\overline{f}"]
		& \mathbb A
	    \\B\arrow[u, "\sqsubseteq_{\E}"]\arrow[ur, swap, "f"]
		\end{tikzcd}
	    \end{center}
		
		\item \emph{$(\K,\E)$-homogeneous}, if any isomorphism between substructures $B_0,B_1 \sqsubseteq_{\E} \mathbb A$ can be extended to an automorphism of $\mathbb A$.
        \begin{center}
		
		\begin{tikzcd}
			B_0 \arrow[d, "\sqsubseteq_{\E}"]  \arrow[r, "\simeq"]
			& B_1 \arrow[d, "\sqsubseteq_{\E}"] 
			\\ \mathbb A \arrow[r, dashed]
			& \mathbb A
		\end{tikzcd}
		
	\end{center}

	\end{itemize}
\end{defin}

\begin{thm}[Fra\"iss\'e Theorem, \cite{fraisse},\cite{jonsson},\cite{kub}]\label{fraissethm}
Let $\lambda$ be an infinite cardinal satisfying $\lambda^{<\lambda}=\lambda$. Let $\K$ be a class of structures in a finite first-order language, and let $\E$ be a distinguished class of embeddings, closed under compositions, and containing all isomorphisms. Let $\K_\lambda$ be a sub-class of $\K$, consisting only of models of size less than $\lambda$. Assume that $(\K_\lambda,\E)$ satisfies the following properties:
\begin{enumerate}
    \item Joint Embedding Property,
    \item Amalgamation Property,
    \item Is closed under $\E$-substructures,
    \item Is $<\lambda$-closed\footnote{The class $(\K_\lambda,\E)$ satisfying $1.-4.$ is often called the \emph{Fra\"iss\'e class of the length $\lambda$}}.
\end{enumerate}

    Then there exists a structure $\mathbb A_\lambda$ that satisfies
    \begin{enumerate}
        \item[a)] $\mathbb A_\lambda=\bigcup\limits_{\alpha<\lambda}A_\alpha$,
        where 
            \begin{itemize}
                \item $A_\alpha \in \K_\lambda$, for every $\alpha<\lambda$,
                \item $A_\alpha \sqsubseteq_{\E} A_\beta$, whenever $\alpha<\beta<\lambda$,
                \item $A_\beta=\bigcup\limits_{\alpha<\beta}A_\alpha$, whenever $\beta<\lambda$ is limit.
            \end{itemize}
         \item[b)] $\mathbb A_\lambda$ is $(\K_\lambda,\E)$-universal,
         \item[c)] $\mathbb A_\lambda$ is $(\K_\lambda,\E)$-injective.
    \end{enumerate}
    Moreover, the properties $1.-4.$ properties determine the structure uniquely $\mathbb A_\lambda$ up to isomorphism, and any structure satisfying them is also $(\K_\lambda,\E)$-homogeneous.
\end{thm}

\begin{proof}
Enumerate as $\{A_\alpha \mid \alpha<\lambda \}$ all (up to isomorphism) structures from $\K_\lambda$. We aim to recursively build an $\sqsubseteq_{\E}$-increasing sequence of structures $F_\alpha \in \K_\lambda$, and then set :
    $$\mathbb A_\lambda=\bigcup_{\alpha<\lambda}F_\alpha.$$
    For bookkeeping purposes, we fix a partition $\{\Phi_\gamma \mid \gamma<\lambda \}$ of the cardinal $\lambda$, consisting of sets of cardinality $\lambda$, such that $\min{\Phi_\gamma}\ge \gamma$, for all $\gamma < \lambda$.
	\vspace{1em}
    
	\textbf{Existence:} Let $F_0=A_0$, and enumerate as $\{g_\gamma \mid \gamma \in \Phi_0 \}$ all $\E$-embeddings $g$, such that $\dom{g}\sqsubseteq_{\E} F_0$ and $\rg{g} \sqsubseteq_{\E} A_\gamma$ for some $\gamma<\lambda$.
	
	\begin{itemize}
	\item  In a successor step $F_\alpha$ is defined, and so is the set $\{g_\gamma \mid \gamma \in \Phi_\alpha \}$. Note that if $\beta$ satisfies $\alpha \in \Phi_\beta$, then $\beta \leq \alpha$, so in particular, $g_\alpha$ is already defined, and $\dom{g_\alpha}\subseteq F_\alpha$. Using the Amalgamation Property, we can find $h$, and $F'_{\alpha+1}\in \K_\lambda$, closing the diagram
	\begin{center}
		
		\begin{tikzcd}
			\dom{g_\alpha} \arrow[d, swap, "g_\alpha"]  \arrow[r,"\sqsubseteq_{\E}"]
			& F_\alpha \arrow[d,dashed] 
			\\ \rg{g_\alpha} \arrow[r, "h", dashed]
			& F'_{\alpha+1}
		\end{tikzcd}
		
	\end{center}
	
	Using the Joint Embedding Property we can enlarge $F'_{\alpha+1}$ to $F_{\alpha+1}$, so that $F_{\alpha+1}$ contains an isomorphic copy of $A_\alpha$. 
    
    We use the set $\Phi_{\alpha+1}$ to index all $\E$-embeddings starting from substructures of $F_{\alpha+1}$.
	
	\item In the limit step, we just set 
    $$F_\alpha=\bigcup_{\gamma<\alpha}F_\gamma,$$
    and use the set $\Phi_{\alpha}$ to index all $\E$-embeddings starting from $\sqsubseteq_{\E}$-substructures of $F_{\alpha}$.
	
	\end{itemize}
	
	Universality of the $\mathbb A_\lambda$ is straightforward, so we leave it for the reader, and proceed to show the injectivity.

    Fix $B\sqsubseteq_{\E} C \in \K_\lambda$, and an $\E$-embedding $i:B\longrightarrow \mathbb A_\lambda$. There is some $\delta<\lambda$ such that $g_\delta = i^{-1}:i[B]\rightarrow B\sqsubseteq_{\E} C$. The recursive construction gives us the commutative diagram: 

    \begin{center}
		
		\begin{tikzcd}
			i[B] \arrow[d, swap, "i^{-1}"]  \arrow[r, "\sqsubseteq_{\E}"]
			& F_\delta \arrow[d, "\sqsubseteq_{\E}"] 
			\\ C \arrow[r, "h", dashed]
			& F'_{\delta+1}
		\end{tikzcd}
		
	\end{center}

    It follows that $h$ and $i$ agree on $B$.
    \vspace{1em}
    
    \textbf{Uniqueness:}
    Suppose that we have another structure $\mathbb A'_\lambda$ satisfying the conclusion of the theorem. In particular there exists a decomposition
    $$\mathbb A'_\lambda = \bigcup_{\alpha<\lambda}{A'_\alpha},$$
    into an $\E$-increasing sequence of $\K_\lambda$-structures.
    Let $f_0:A_0\longrightarrow A'_0$ be given. We will recursively define an increasing sequence of $\E$-embeddings $f_\alpha$ such that 
    $$f=\bigcup_{\alpha<\lambda}f_\alpha$$ will be an isomorphism from $\mathbb A_\lambda$ onto $\mathbb A'_\lambda$.
	\begin{itemize}
		\item Suppose $\alpha$ is an even successor ordinal and $f_\alpha$ is defined. The cardinal $\lambda$ is regular, so there exists $\beta>\alpha$ such that $\dom{f_\alpha}\sqsubseteq_{\E} A_\beta$. Since $\mathbb A'$ is $(\K_\lambda, \E)$-injective, we can find $f_{\alpha+1}$ closing the diagram
		
		\begin{center}
			\begin{tikzcd}[column sep=4 pt]
			\dom{f_\alpha} \arrow[d,swap, "f_\alpha"] 
			& \sqsubseteq_{\E} 
			& A_\beta \arrow[d,dashed, "f_{\alpha+1}"] 
			\\ \rg{f_\alpha} 
			& \sqsubseteq_{\E}
			& \mathbb{A}'_\lambda
		\end{tikzcd}
		\end{center}
		In particular, $A_{\alpha+1} \subseteq \dom(f_{\alpha+1})$.
		\item Suppose $\alpha$ is an odd successor ordinal and $f_\alpha$ is defined. There exists $\beta>\alpha$ such that $\rg{f_\alpha}\sqsubseteq_{\E} B_\beta$. Since $\mathbb A$ is $(\K_\lambda,\E)$-injective, we can find $f_{\alpha+1}$ closing the diagram
		
		\begin{center}
		\begin{tikzcd}[column sep=4 pt]
			\dom{f_\alpha} 
			& \sqsubseteq_{\E} 
			& \mathbb A_\lambda 
			\\ \rg{f_\alpha} \arrow[u, "f_\alpha^{-1}"] 
			& \sqsubseteq_{\E}
			& A'_\beta \arrow[u,dashed,swap, "f_{\alpha+1}^{-1}"] 
		\end{tikzcd}
		\end{center}
		In particular, $A'_{\alpha+1} \subseteq \rg(f_{\alpha+1})$.
		\item If $\alpha$ is limit, we set 
        $$f_\alpha=\bigcup_{\gamma<\alpha}f_\gamma.$$
		
	\end{itemize}
	
	It is evident that $f=\displaystyle{\bigcup_{\alpha<\lambda}f_\alpha}$ is an isomorphism.

    \textbf{Homogeneity:} To see that any structure satisfying the conditions $1.-4.$ is $(\K_\lambda,\E)$-homogeneous, note that in the proof of the uniqueness we started the construction from an arbitrary partial isomorphism $f_0$. Therefore we can just take $\mathbb A'_\lambda = \mathbb A_\lambda$ and follow the same construction.
\end{proof}

This theorem justifies the following important definition.
\begin{defin}
    The structure $\mathbb A_\lambda$ from Theorem \ref{fraissethm} is called the \emph{Fra\"iss\'e limit} of the class $(\K_\lambda,\E)$.
\end{defin}

\subsection{Classes of Boolean algebras}

At this point we restrict our attention to the class of (separative) partial orders, together with regular embeddings. Most of partial orders under consideration will be of the form $B^+$ for some Boolean algebra $B$.

\begin{defin}\label{classes} Given an uncountable cardinal $\lambda$, we introduce the following classes of Boolean algebras:
\begin{itemize}
    \item The class $\Bool_\lambda$ consists of all Boolean algebras of size less than $\lambda$, together with regular embeddings.
    \item The class $\BoolSeq_\lambda$ consists of all Boolean algebras $A$ of the form 
    $$A=\bigcup_{\alpha<\lambda}A_\alpha,$$
    where 
    \begin{itemize}
    \item $A_\alpha \in \Bool_\lambda$, for each $\alpha<\lambda$;
    \item $A_\alpha\sqsubseteq A_\beta$, for each $\alpha<\beta<\lambda$;
    \item $\bigcup\limits_{\alpha<\beta}A_\alpha=A_\beta$, for each $\alpha<\beta<\lambda$.
    \end{itemize}
    The embeddings of $\BoolSeq_\lambda$ are regular embeddings.
\end{itemize}
\end{defin}

Our goal is to characterize the Fra\"iss\'e limit of the class $\Bool_\lambda$, in the case where $\lambda$ is strongly inaccessible.

\section{Collapsing algebras}
The following theorem is a well-known characterization of the collapsing algebra.
\begin{thm}[Lemma 26.7, \cite{jech}] \label{thm1}
	For any infinite cardinal $\kappa$, there exists a unique up to isomorphism complete Boolean algebra $\coll$ that satisfies the following two properties:
	\begin{itemize}
	\item $\dens(\coll)=\kappa$,
	\item $\coll \Vdash |\kappa|=\omega$.
	\end{itemize}
	The algebra $\coll$ has size $2^\kappa$.
\end{thm}

From these conditions it is easy to find a specific representation of the collapsing algebra $\coll$. For instance, $\coll$ is isomorphic to the completion of the partial order

$$B_\kappa := \{p:\dom(p)\longrightarrow \kappa \mid \dom(p)\in [\omega]^{<\omega}\},$$
with the order $p\leq q \iff p \supseteq q$.

\begin{cor}[Kripke, Thm. 26.8, \cite{jech}]\label{cor1}
	
	If $B\in \Bool_{\kappa^+}$, then
	$$\overline{B\oplus \coll} \simeq \coll.$$
	In particular, the algebra $\coll$ $\Bool_{\kappa^+}$-universal.
\end{cor}

An immediate consequence is that $\overline{\operatorname{Coll}(\omega,\delta_0) \oplus \operatorname{Coll}(\omega,\delta_1)} \simeq \operatorname{Coll}(\omega,\delta_1)$,
whenever $\delta_0 < \delta_1$. As a matter of fact, even more is true.

\begin{thm} \label{mainthm}
	Let $\kappa$ be an uncountable cardinal, and $B$ be a Boolean algebra of density less than $\kappa$. Let
    $$i:B \longrightarrow \operatorname{Coll}(\omega,\kappa)$$
    be a regular embedding. Then there exists an isomorphism 
	$$h:\operatorname{Coll}(\omega,\kappa) \longrightarrow \overline{B \oplus \operatorname{Coll}(\omega,\kappa)},$$
	such that $h\circ i [B]$ is the canonical embedding of $B$ into $\overline{B \oplus \operatorname{Coll}(\omega,\kappa)}$.
    \begin{center}
        \vspace{1cm}
		\begin{tikzcd}[row sep=2cm]
		\operatorname{Coll}(\omega,\kappa)\arrow[r, "h"]
		& \overline{B\oplus \operatorname{Coll}(\omega,\kappa)}
        \\ B\arrow[u, "i"]\arrow[ur, "\eta_B"]
		\end{tikzcd}
        \vspace{1cm}
	\end{center}
\end{thm}
The proof will rely on a Lemma.
\begin{lem} \label{forcinglem}
	Suppose that for some forcing $\mathbb{P}$, we have
	$$ \mathbb{P} \Vdash \dot{f}: \check{\mathbb{Q}}_0 \longrightarrow \dot{\mathbb{Q}}_1 \text{ is an isomorphism of forcings}.$$
	Then a function $\bar{f}: \mathbb{P} \ast \check{\mathbb{Q}}_0 \longrightarrow \mathbb{P} \ast \dot{\mathbb{Q}}_1$
	given by
	$$(p,\check{q})\mapsto (p,\dot{f}(\check{q}))$$
	is an isomorphism onto a dense subset.	
\end{lem}

\begin{proof}
	Verification that $\bar{f}$ preserves ordering and incompatibility is straightforward. We will check that the conditions of the form $(p,\dot{f}(\check{q}))$ constitute a dense subset of 
	$\mathbb{P} \ast \dot{\mathbb{Q}}_1$. 
	\par Fix $(p,\dot{q})\in \mathbb{P} \ast \dot{\mathbb{Q}}_1$. There is a $\mathbb{P}$-name $\dot{q}_0$ such that
	$$p \Vdash \dot{f}(\dot{q}_0)=\dot{q}.$$
	Let $p_0\le p$ be a condition deciding $\dot{q}_0$, i.e. 
	$$p_0 \Vdash \dot{q}_0=\check{q}_0.$$
	Then it is straightforward that $(p_0,\dot{f}(\check{q}_0)) \le (p,\dot{q}).$
\end{proof}

\begin{proof}[Proof of Theorem \ref{mainthm}]
	Without loss of generality we can assume that $i$ is the identity inclusion. Let $B_0 \subseteq B_1$ be some dense sub-algebras of $B$ and  $\operatorname{Coll}(\omega,\kappa)$ respectively, such that $|B_0|<\kappa$, and $|B_1|=\kappa$. Since a dense sub-algebra is always regular, it follows that $B_0\sqsubseteq B_1$. Then
	
	$$B_1^+ \simeq B_0^+ \ast (B_1^+:\dot{G}) \sqsubseteq B_0^+\ast (\overline{B_1^+:\dot{G}})^+,$$
	where $\dot{G}$ is a $B_0^+$-name for a generic filter.
	Note that
	$$B_0^+ \Vdash |(B_1^+:\dot{G})|=\kappa,$$
	and
	$$B_0^+ \Vdash \text{``} (B_1^+:\dot{G}) \Vdash |\kappa|=\omega. \text{"}$$
	It follows from Theorem \ref{thm1} that $B_0^+$ forces $\overline{B_1^+:\dot{G}}$ and 
	$\operatorname{Coll}(\omega,\kappa)$ to be isomorphic. Applying Lemma \ref{forcinglem}, we therefore obtain a sequence of dense embeddings
	$$B_0^+\ast (\overline{B_1^+:\dot{G}})^+ \longrightarrow B_0^+ \ast \operatorname{Coll}(\omega,\kappa)^+ \sqsubseteq B_0^+\times \operatorname{Coll}(\omega,\kappa)^+.$$
	
	This way we got a commutative diagram of regular embeddings
	
	\begin{center}
		\begin{tikzcd}
		B_1^+\arrow[r]
		& B_0^+\times \operatorname{Coll}(\omega,\kappa)^+
	\\B_0^+\arrow[u, "\sqsubseteq"]\arrow[ur, swap, "\eta_{B_0^+}"]
		\end{tikzcd}
	\end{center}
	where $\eta_{B_0^+}$ is the canonical embedding of $B_0^+$
	into $B_0^+\times \operatorname{Coll}(\omega,\kappa)^+$. Given that $B_0$ is a dense sub-algebra of $B$, after passing to completions we obtain

	\begin{center}
		\begin{tikzcd}
		\operatorname{Coll}(\omega,\kappa)\arrow[r]
		&\overline{B \oplus \operatorname{Coll}(\omega,\kappa)}
		\\
		\overline{B}\arrow[u, "\sqsubseteq"]\arrow[ur, swap, "\eta_{\bar{B}}"]
		\end{tikzcd}
	\end{center}
This concludes the proof.
\end{proof}

\begin{cor}\label{coramalgamation}
	If $\lambda$ is strongly inaccessible, then $\Bool_\lambda$ has the Amalgamation Property.
\end{cor}
\begin{proof}
    Consider a diagram in the category $\Bool_\lambda$:
    \begin{center}
			\begin{tikzcd}
				&   B_1 
				&
				& \\
				B_0 \ar[ur, "i_1"] \ar[dr, "i_2"]
				&
				& 
				\\
				&    B_2
				&
				&
			\end{tikzcd}
		\end{center}
        By Corollary \ref{cor1}, every Boolean algebra can be embedded into a collapsing algebra of the same density, so we can assume -- without loss of generality -- that $B_1=\col(\omega,\delta_1)$, and $B_2=\col(\omega,\delta_2)$, for some cardinals $\delta_1,\delta_2<\lambda$.
        By Theorem \ref{mainthm} we can extend $i_1$ and $i_2$ to canonical embeddings
        $$h_1\circ i_1:B_0 \longrightarrow \overline{B_0\oplus \col(\omega,\delta_1)},$$
        and
        $$h_2\circ i_2:B_0 \longrightarrow \overline{B_0\oplus \col(\omega,\delta_2)}.$$

        Let $\eta_1$ and $\eta_2$ be the canonical embeddings:
        $$\eta_1:\overline{B_0\oplus \col(\omega,\delta_1)}\longrightarrow \overline{B_0\oplus \col(\omega,\delta_1)\oplus \col(\omega,\delta_2)},$$
        and
        $$\eta_2:\overline{B_0\oplus \col(\omega,\delta_2)}\longrightarrow \overline{B_0\oplus \col(\omega,\delta_1)\oplus \col(\omega,\delta_2)}.$$

        Then both compositions in the following diagram are canonical, so the diagram commutes.
        
        \begin{center}
			\begin{tikzcd}
				&   \col(\omega,\delta_1) \ar[r, "h_1"] & \overline{B_0\oplus \col(\omega,\delta_1)} \ar[dr, "\eta_1"]
				&
				& \\
				B_0 \ar[ur, "i_1"] \ar[dr, "i_2"]  & & & \overline{B_0\oplus \col(\omega,\delta_1) \oplus \col(\omega,\delta_2)}
				&
				& 
				\\
				&    \col(\omega,\delta_2) \ar[r, "h_2"] & \overline{B_0\oplus \col(\omega,\delta_2)}  \ar[ur, "\eta_2"]
				&
				&
			\end{tikzcd}
		\end{center}
        
\end{proof}

Another corollary is a minor refinement of Lemma 26.9 from \cite{jech}, where complete embeddings are considered.

\begin{cor}[injectivity of the Collapsing Algebras] \label{injectivityColAlg}
	Let $\delta$ be any cardinal. For every pair of Boolean algebras $B \sqsubseteq C$, where $\dens(C)<\delta$, each regular embedding $$i:B \longrightarrow \operatorname{Coll}(\omega,\delta)$$ can be extended to a regular embedding $$\bar{i}:C \longrightarrow \operatorname{Coll}(\omega,\delta).$$
\end{cor}

\begin{proof}
	By Theorem \ref{mainthm}, $i$ is equivalent to the composition:
	$$B \sqsubseteq C \longrightarrow \overline{C\oplus \operatorname{Coll}(\omega,\delta)},$$
	and this embedding can clearly be extended to $C$.
\end{proof}

\begin{thm}\label{uniquenessthm}
    If $\lambda$ is strongly inaccessible, then $\Bool_\lambda$ is a Fra\"iss\'e class. In particular, there exists a unique up to isomorphism Boolean algebra $\blim \in \BoolSeq_\lambda$ that is $\Bool_\lambda$-injective. Moreover, the algebra $\blim$ is $\Bool_\lambda$-homogeneous.
\end{thm}
\begin{proof}
	Almost all conditions required by Theorem \ref{fraissethm} are straightforward. The single exception is the Amalgamation Property, which is Corollary \ref{coramalgamation}.
\end{proof}

Our next goal is to find a representation of the algebra $\blim.$ By unpacking the definition of $\Bool_\lambda$-homogeneity, we obtain:

\begin{cor}
	For each pair of isomorphic regular sub-algebras $B_0,B_1\sqsubseteq \blim$, every isomorphism
	$f:B_0 \longrightarrow B_1$ can be extended to an automorphism $\bar{f}:\blim \longrightarrow \blim$. 
\end{cor}

This is a well-known property of the the L\'evy collapse, that is proven in Theorem 26.12 \cite{jech}. In fact, we will see that $\blim$ contains the L\'evy collapse as a dense subset.

\begin{thm} \label{structurethm}
	The algebra $\blim$ from Theorem \ref{uniquenessthm} is isomorphic to any of the following:
    
    \begin{enumerate}
    \item Any $\sqsubseteq$-increasing union
    $$\bigcup_{\alpha<\lambda} \operatorname{Coll}(\omega,|\alpha|).$$
	\item The free product 
	$$\bigoplus_{\delta<\lambda} \operatorname{Coll}(\omega,|\delta|).$$	
    
    \end{enumerate}
\end{thm}

\begin{proof}
	Notice that for any uncountable cardinal $\kappa$ we have
    $$\bigoplus_{\delta\leq \kappa} \operatorname{Coll}(\omega,\delta)\simeq\collkappa,$$
    by Corollary \ref{cor1}. Therefore $2.$ is a special case of $1.$ so we will focus on the proof of $1.$

    First, we will verify that the union
    $$\bigcup_{\alpha<\lambda} \operatorname{Coll}(\omega,|\alpha|)$$
    is indeed in the class $\Bool_\lambda$. This is almost evident, as it is defined as an increasing chain, but unfortunately not a continuous one. In order to refine this union to a continuous one, we will work with the set $S \subseteq \lambda$, consisting of all singular cardinals of countable cofinality. Evidently
    $$\bigcup_{\alpha<\lambda} \operatorname{Coll}(\omega,|\alpha|)=\bigcup_{\alpha\in S} \operatorname{Coll}(\omega,|\alpha|),$$
    so it is sufficient to argue that
    
    $$\bigcup_{\alpha \in S\cap \delta} \operatorname{Coll}(\omega,|\alpha|)\simeq\col(\omega,\delta)$$  
    for every $\delta \in S$. 
    Since $\operatorname{cof}(\delta)=\omega$, if every cardinal below $\delta$ becomes countable, then so does $\delta$ itself. Moreover, we have
    
    $$\dens(\bigcup_{\alpha<\delta} \operatorname{Coll}(\omega,|\alpha|)\leq \bigcup\limits_{\alpha<\delta}|\alpha|=\delta.$$ 
    The claim follows from Theorem \ref{thm1}.

    In the class of Boolean algebras, universality clearly follows from injectivity, so we need only to prove that the union is $\Bool_\lambda$-injective. This, however, follows immediately from the injectivity of the Collapsing Algebras (Corollary \ref{injectivityColAlg}).
\end{proof}

The "absorption property" extends to the class $\BoolSeq_\lambda$:

\begin{prop}\label{absorption}
    If $B\in \BoolSeq_\lambda$, then $B\oplus \blim\simeq \blim$.
\end{prop}
\begin{proof}
    Fix a continuous decomposition $B=\bigcup\limits_{\alpha<\lambda}B_\alpha$, for some Boolean algebras $B_\alpha \in \Bool_\lambda$. Then we have
    $$B \oplus \blim \simeq (\bigcup_{\alpha<\lambda}B_\alpha) \oplus (\bigcup_{\alpha<\lambda} \operatorname{Coll}(\omega,|B_\alpha|)) =$$
	$$\bigcup_{\alpha<\lambda}B_\alpha \oplus \operatorname{Coll}(\omega,|B_\alpha|)=\bigcup_{\alpha<\lambda}\operatorname{Coll}(\omega,|B_\alpha|)\simeq \blim.$$
\end{proof}

\subsection{L\'evy collapse}

\begin{defin}
    The \emph{L\'evy collapse} of an uncountable cardinal $\lambda$ is the partial order

    \begin{align*} \colllambda=\{p:\dom(p)\longrightarrow \lambda\mid \dom(p) \in [\lambda \times \omega]^{<\omega},\\
    \forall (\alpha,n)\in\dom(p) \; p(\alpha,n)<\alpha\}.
    \end{align*}
\end{defin}

Directly from the definition, we can decompose the L\'evy collapse as an increasing union of regular subsets that collapse bigger and bigger cardinals, specifically
$$\colllambda=\bigcup_{\beta<\lambda} C_\beta,$$
where 
    \begin{align*} C_\beta=\{p:\dom(p)\longrightarrow \beta\mid            \dom(p) \in [\beta \times \omega]^{<\omega},\;
     \forall (\alpha,n)\in\dom(p) \; p(\alpha,n)<\alpha\}.
    \end{align*}

\begin{prop}\label{prop1}
    If $\kappa<\lambda$ is an uncountable cardinal, then $C_\kappa=\col(\omega,<\kappa)$, and $\overline{C_{\kappa+1}}\simeq \col(\omega,\kappa)$.
\end{prop}
\begin{proof}
    The fact that $C_\kappa=\col(\omega,<\kappa)$ follows right from the definition. To see that $\overline{C_{\kappa+1}}\simeq \col(\omega,\kappa)$, we use Theorem \ref{thm1}. Evidently $|C_{\kappa+1}|=\kappa$, so it is sufficient to show that $C_{\kappa+1}$ collapses $\kappa$ to a countable ordinal. To see this, let $\dot{G}$ be a $C_{\kappa+1}$-name for a generic filter. Then a standard density argument shows that
    $$C_{\kappa+1}\Vdash \text{``}\bigcup \{\langle n,p(\kappa,n)\rangle \mid p \in \dot{G}\}\text{ is a function from $\omega$ onto $\kappa$}.\text{"}$$
\end{proof}

\begin{thm}\label{thm2}

If $\lambda$ is strongly inaccessible\footnote{For this argument it is relevant only that $\lambda$ is uncountable and limit, but $\mathbb \blim$ is defined only for strongly inaccessible $\lambda$.}, then there is an isomorphic copy of $\blim$ between $\colllambda$ and $\overline{\colllambda}$. In particular $\overline{\blim}\simeq \overline{\colllambda}$.
\end{thm}
\begin{proof}
    Right from the definition we see that $C_\beta \subseteq \colllambda$, and in fact $C_\beta \sqsubseteq \colllambda$.  Therefore we can assume that each completion $\overline{C_\beta}$ is taken inside $\overline{\colllambda}$. Then we have a sequence of inclusions
    $$\colllambda = \bigcup_{\beta<\lambda}C_\beta \s \bigcup_{\beta<\lambda}\overline{C_\beta} \s \overline{\colllambda}.$$
    Since $\lambda$ is limit,
    $$\bigcup_{\beta<\lambda}\overline{C_\beta}=\bigcup_{\kappa<\lambda}\overline{C_{\kappa+1}}\simeq \bigcup_{\kappa <\lambda} \coll,$$ 
    where $\kappa$ runs over uncountable cardinals. By Theorem \ref{structurethm}, this union is isomorphic to $\blim$.

\end{proof}

\section{Universality of the group $\aut{\blim}$}

\begin{prop} \label{prop1}
	For any pair of Boolean algebras $B,C$, there exists an embedding
	$$\aut{B} \longrightarrow \aut{B \oplus C}.$$
\end{prop}

\begin{proof}
	Let us denote by $X_B$ and $X_C$ the Stone spaces of $B$ and $C$ respectively.
	Then $B \oplus C$ is the algebra of clopen subsets of $X_B \times X_C$. If $h:X_B \longrightarrow X_B$ is a homeomorphism, we can extend it to
    
	$$\bar{h}=h \times \operatorname{id} :X_B \times X_C \longrightarrow X_B \times X_C.$$
	
	The mapping $\bar{h}$ induces an automorphism of $B\oplus C$ via the Stone duality. More precisely, the function given by
    
	$$F \mapsto \bar{h}^{-1}[F],$$
    
	for $F\in \operatorname{clop}(X_B \times X_C)$ is an automorphism of $B\oplus C$.

    The mapping $h  \mapsto \bar{h}$ is a group homomorphism, that is moreover injective.	
\end{proof}

\begin{thm}\label{topuniversality}
	If $B\in \BoolSeq_\lambda$, for a strongly inaccessible $\lambda$, then there exists a group embedding
	$$\aut{B} \longrightarrow \aut{\blim}.$$
\end{thm}

\begin{proof}
	
	An immediate consequence of Proposition \ref{prop1} and Proposition \ref{absorption}.
	
\end{proof}

It would be desirable to have a better understanding of the class $\BoolSeq_\lambda$, at least for strongly inaccessible cardinals. We will prove that this class contains all Boolean algebras that do not have antichains of size $\lambda$.  Recall that the saturation of a Boolean algebra $B$, denoted by $\operatorname{sat}{B}$, is the smallest cardinal $\delta$, such that $B$ has no antichain of size $\delta$.

\begin{prop}
	If $\kappa$ is an uncountable regular cardinal and $\operatorname{sat}{B}<\kappa$, then $B\in \BoolSeq_\kappa$.
\end{prop}
\begin{proof}
	
We can assume that $|B|=\kappa$, for if $|B|<\kappa$ the result trivial. Let $\delta=\operatorname{Sat}{B}$. Consider a function 
$$W:[B]^{< \delta}\longrightarrow B^+,$$
such that for any set $D \in [B]^{< \delta}$, $W(D)$ is disjoint with every element of $D$, whenever such element exists. We will say that a set $E \subseteq B$ is \emph{closed} under $W$, if for any $D \in [E]^{< \delta}$

$$W(D) \in E.$$

By the regularity of $\kappa$, we can decompose $B$ into an increasing chain of sub-algebras closed under $W$. It is easy to see that every such sub-algebra must be regular.
\end{proof}

\section{The collapsing algebra is not in $\BoolSeq$}

Answering a question of Adam Barto\v{s}, we prove that collapsing algebra $\coll$ is not in the class $\BoolSeq_\kappa$ for any uncountable regular $\kappa$. In fact, for complete Boolean algebras this follows immediately from Corollary 16.10 in \cite{jech}, as any algebra in $\BoolSeq_\kappa$ has to satisfy the $\kappa$-chain condition, whilst $\coll$ doesn't.

We provide an independent proof not relying on the theory of finite support iterations.

\begin{thm}[\cite{jech}] \label{lastthm}
    Assume that $\kappa$ is a cardinal of uncountable cofinality, and we are given a sequence of Boolean Algebras $\{B_\alpha \mid \alpha \le \kappa \}$, satisfying:
\begin{enumerate}
    \item for all $\alpha<\beta\le\kappa$, $B_\alpha$ is a regular sub-algebra of $B_\beta$,
    \item for all $\alpha<\kappa$, $\dens(B_\alpha)<\kappa$,
    \item for each $\beta \in \acc(\kappa)$, the union $\displaystyle\bigcup_{\alpha<\beta}B_\alpha$ is dense in $B_\beta$.
\end{enumerate}
Then $B_\kappa$ is not isomorphic to $\collkappa$.
\end{thm}

\begin{proof}
    We know that $\collkappa$ can be represented as the completion of the $\kappa$-splitting tree of the height $\omega$. Therefore it is sufficient to show that there is no dense embedding of the tree $\kappatree$ into $B_\kappa^+$. Towards contradiction, let us assume that 
    $$\phi: \kappatree \longrightarrow B_\kappa^+$$
    is such embedding. Consider a continuous $\in$-chain $\{M_\alpha \mid \alpha<\kappa\}$ of elementary submodels of $\operatorname{H}(\theta)$, for sufficiently large regular cardinal $\theta$, containing $\phi,\kappa, \{B_\alpha \mid \alpha<\kappa\}$, satisfying moreover:
    \begin{enumerate}
        \item[(1')] for each $\alpha<\kappa$, $|M_\alpha|=|\alpha|+\omega$,
        \item[(2')] for each $\alpha<\kappa$, $\alpha\s M_\alpha$.
    \end{enumerate}

    By the Clause (3), the function $\alpha \mapsto \dens (B_\alpha)$ is continuous, so we can find $\delta\in \acc(\kappa)$ such that
    \begin{enumerate}
        \item[(1'')] $\kappa \cap M_\delta = \delta$,
        \item[(2'')] $\sup(\{\alpha<\delta \mid \dens(B_\alpha)=|\alpha|\})=\delta.$
    \end{enumerate}

   We aim to show that $B_\delta \cap M_\delta$ is dense in $B_\delta$. By item (2'') we can find an arbitrarily large $\alpha<\delta$ such that $B_\alpha^+$ has a dense subset $D_\alpha$ of size $|\alpha|$. Be elementarity, we can choose $D_\alpha$ from the model $M_\delta$, and given that $|\alpha|+1 \s M_\delta$, we conclude that $D_\alpha \s M_\delta$. Since $\alpha$ was arbitrary, it follows that the union $\displaystyle\bigcup_{\alpha<\delta}B_\alpha \cap M_\delta$ is dense in $\displaystyle\bigcup_{\alpha<\delta}B_\alpha$, and so in $B_\delta$.

   We claim that $\phi[^{<\omega}\delta]$ is dense in $B_\delta$. To see this, pick any non-zero $p\in B_\delta$. We can find non-zero $q\le p$ in $B_\delta\cap M_\delta$.  By elementarity
    $$M_\delta\models \exists\; t \in \kappatree \quad \phi(t) \le q.$$
    Therefore we can find $t \in \kappatree$ witnessing $\phi(t)\le p$ inside $\kappatree \cap M_\delta = ^{<\omega}\delta$. Let $A\s ^{<\omega}\delta$ be a maximal antichain. Since $\phi$ is a dense embedding, also $\phi[A]$ is maximal in $B_\delta$, and therefore in $B_\kappa$. However, this clearly can't be the case, given that for any $s \in \kappa \setminus \delta$, the  one-element sequence $\langle s \rangle \in \kappatree$ is incompatible with any element of $A$, and for this reason $\phi[A] \cup \{ \phi(\langle s \rangle)\}$ is an antichain extending $\phi[A]$, contrary to the maximality of $\phi[A]$. This shows that $\phi$ can't be a dense embedding, and concludes the proof.
\end{proof}

\section{The open question}
Theorem \ref{mainthm} essentially states that $\coll$ is homogeneous for the class $\Bool_\kappa$. It is natural to conjecture that $\coll$ itself is a (generalized?) Fra\"iss\'e limit of some sort. By Theorem \ref{lastthm} it cannot be a Fra\"iss\'e limit of the length $\kappa$ of any class of Boolean algebras.

\begin{prob}
    Represent the collapsing algebra $\coll$ as a Fra\"iss\'e limit.
\end{prob}

\section*{Acknowledgements}
The author was supported by the GA\v{C}R project EXPRO 20-31529X and RVO: 67985840.

\end{document}